\documentclass[a4paper,11pt]{amsart}
\usepackage[utf8]{inputenc}
\usepackage{amsmath,amsthm,amsfonts,amssymb}
\usepackage{color}
\usepackage{xcolor}
\usepackage{hyperref}
\hypersetup{
    unicode=false,          % non-Latin characters in Acrobat’s bookmarks
    pdftoolbar=true,        % show Acrobat’s toolbar
    pdfmenubar=true,        % show Acrobat’s menu?
    pdffitwindow=false,     % window fit to page when opened
    pdfstartview={FitH},    % fits the width of the page to the window
    pdftitle={},    % title
    pdfauthor={Soumendu Sundar Mukherjee},     % author
    pdfsubject={Mathematics},   % subject of the document
    pdfcreator={Creator},   % creator of the document
    pdfproducer={Producer}, % producer of the document
    pdfkeywords={keyword1} {key2} {key3}, % list of keywords
    pdfnewwindow=true,      % links in new window
    colorlinks=true,       % false: boxed links; true: colored links
    linktoc=page,
    linkcolor=blue,          % color of internal links
    citecolor=magenta,        % color of links to bibliography
    filecolor=magenta,      % color of file links
    urlcolor=cyan,           % color of external links
    linkbordercolor={1 1 1},
    citebordercolor={0 1 0},
    urlbordercolor={0 1 1}
}
\numberwithin{equation}{section}

\newtheorem{prop}{Proposition}[section]
\theoremstyle{plain}
\newtheorem{theorem}{Theorem}[section]
\theoremstyle{plain}
\newtheorem{lemma}{Lemma}[section]
\theoremstyle{plain}
\newtheorem{cor}{Corollary}[section]
\theoremstyle{plain}
\newtheorem{exm}{Example}[section]
\theoremstyle{remark}
\newtheorem{rem}{Remark}[section]
\theoremstyle{remark}

\title{An Approximation Inequality For Continued Radicals and Power Forms}
\author{Soumendu Sundar Mukherjee}
\address{Masters Student\\
Indian Statistical Institute\\
Kolkata 700108}
\email{soumendu041@gmail.com}
\thanks{Supported by the KVPY fellowship, funded by the Department of Science \& Technology, Government of India}
\date{24 October, 2013}

\begin{document}

\begin{abstract}
In this article we derive an approximation inequality for continued radicals, generalizing an inequality of Herschfeld for continued square roots to arbitrary radicals, which is useful in exploring convergence issues and obtaining convergence rates. In fact, we generalize this inequality further to encompass the more general continued power forms. We demonstrate the use of this inequality by obtaining estimates for the convergence rates of several continued radicals including the famous Ramanujan radical. 
\end{abstract}

\maketitle
\tableofcontents

\section{Introduction}
Continued radicals aka infinitely nested radicals are infinite constructs of the forms
\begin{equation}\label{right}
\sqrt[r_1]{a_1+\sqrt[r_2]{a_2+\sqrt[r_3]{a_3+\cdots}}}
\end{equation}
or
\begin{equation}\label{left}
\cdots\sqrt[r_3]{a_3+\sqrt[r_2]{a_2+\sqrt[r_1]{a_1}}},
\end{equation}
the first one being a ``right" continued radical and the second one ``left". We shall call $\sqrt[r_1]{a_1+\sqrt[r_2]{a_2+\sqrt[r_3]{a_3+\cdots +\sqrt[r_n]{a_n}}}}$ or $\sqrt[r_n]{a_n+\cdots\sqrt[r_3]{a_3+\sqrt[r_2]{a_2+\sqrt[r_1]{a_1}}}}$ the $n^{th}$ \emph{approximants} of the radicals (\ref{right}) or (\ref{left}). For left radicals, there is a recurrence relation connecting the successive approximants. Right radicals lack this property and to compute their approximants one needs to start from the tail-end, i.e., $\sqrt[r_n]{a_n}$ at each step (as Jones \cite{jones1} puts it, these have an end but no beginning!). This and the inherent non-linearity complicates their analysis considerably and perhaps for that reason continued radicals have not been investigated as thoroughly as their other infinite counterparts like infinite series, infinite products or continued fractions. Moreover, most of the existing results assume the $a_n$ to be non-negative (in order to remain in the territory of real numbers!). Continued radicals with negative or more generally complex $a_n$ are more challenging to analyze. See, for example, \cite{borbarra,efth1,efth2,sizer2,zimho}.

 In this article we shall consider continued right radicals with non-negative input sequence $\{a_n\}$. Convergence questions about such radicals appear in P\'{o}lya and Szeg\H{o}'s classic \cite{szegpol}. To fix notations let
\begin{equation}\label{sqrad}
u_{n}=\sqrt{a_{1}+\sqrt{a_2+\cdots+\sqrt{a_n}}}
\end{equation}
be the $n^{th}$ approximant to the the right continued square root
\begin{equation}
\sqrt{a_1+\sqrt{a_2+\sqrt{a_3+\cdots}}}
\end{equation} 
with non-negative input sequence $\{a_n\}$. More generally let
\begin{equation}\label{genrad}
v_n=\sqrt[r_1]{a_{1}+\sqrt[r_2]{a_2+\cdots+\sqrt[r_n]{a_n}}}
\end{equation}
be the $n^{th}$ approximant to the general continued right radical
\begin{equation}
\sqrt[r_1]{a_{1}+\sqrt[r_2]{a_2+\sqrt[r_3]{a_3+\cdots}}}
\end{equation}
where again the inputs $a_n$ are non-negative reals and $r_n$ are positive integers. We make a note that  $\{v_n\}$ is a monotonic non-decreasing sequence of non-negative reals. Problem $162$ of \cite{szegpol} considers the sequence $\{u_n\}$ with $a_n>0$ for all $n$. It states the following criterion for convergence/divergence. 
\begin{prop}\label{pol1}
Let
\[
\limsup\limits_{n\rightarrow \infty} \frac{\log \log a_n}{n}=\alpha.
\]
Then $\{u_n\}$ converges if $\alpha >2$ and diverges if $\alpha <2$.
\end{prop}
Here if $a_n \leqslant 1$ for some $n$ then P\'{o}lya and Szeg\H{o}'s convention is to interpret $\frac{\log \log a_n}{n}$ as $-\infty$. 

Problem $163$, a continuation of Problem $162$, asks the reader to show the following
\begin{prop}\label{pol2}
The sequence $\{u_n\}$ converges if the series
\[
\sum\limits_{n=1}^{\infty}2^{-n}a_n(a_1\cdots a_n)^{-\frac{1}{2}}
\]
converges.
\end{prop}
Proposition \ref{pol2} follows from the following inequality (see \cite{h'feld, szegpol}).
\begin{prop}[P\'{o}lya-Szeg\H{o}]\label{ineq1}
For each $n\geqslant 1$ we have
\begin{equation}
u_{n+1}-u_n\leqslant \frac{\sqrt{a_{n+1}}}{2^n \sqrt{a_1}\sqrt{a_2}\cdots \sqrt{a_n}}.
\end{equation}
\end{prop}

 Herschfeld \cite{h'feld}, in a comprehensive study, established necessary and sufficient conditions for the convergence of both right and left continued square roots. As we are interested with right radicals only, we reproduce Herschfeld's criterion for right square roots.
\begin{theorem}[Herschfeld, 1935]\label{h'feld1}
The sequence $\{u_n\}$ converges if and only if
\[
 \limsup_{n\rightarrow\infty}a_n^{2^{-n}}<\infty.
\]
\end{theorem}
It is a good exercise in classical real analysis to establish Theorem \ref{h'feld1} and show that it encompasses the P\'{o}lya-Szeg\H{o} criterion in Proposition \ref{pol1}. See \cite{borbarra, jones1, jones2, sizer} for other interesting and useful convergence criteria, including a rediscovery of Herschfeld's criterion in \cite{sizer}.

In \cite{h'feld} Herschfeld derived an inequality stronger than the one in Proposition \ref{ineq1}. We present a slightly modified statement here.
\begin{theorem}[Herschfeld, 1935]\label{ineq2}
We have for each $n\geqslant 1$,
\begin{equation}
u_{n+1}-u_n\leqslant \frac{\sqrt{a_{n+1}}}{2^n \sqrt{a_1+\cdots+\sqrt{a_n}}\sqrt{a_2+\cdots+\sqrt{a_n}}\cdots \sqrt{a_n}}
\end{equation}
where, of course, all the $a_n$ are assumed to be positive.
\end{theorem}
From Theorem \ref{ineq2} one readily obtains Proposition \ref{ineq1} because $\sqrt{a_i+\cdots\sqrt{a_n}}\geqslant \sqrt{a_i}$ for each $i=1,\ldots,n$. Herschfeld used the following elementary inequality repeatedly to obtain the inequality in Theorem \ref{ineq2}:
\begin{equation}
\sqrt{a+x}\leqslant \sqrt{a} + \frac{x}{2\sqrt{a}} \text{, where $a>0$ and $x\geqslant 0$.}
\end{equation} 
Theorem \ref{ineq2} may be used to infer about the rate of convergence of the sequence $\{u_n\}$.

\section{An Inequality for General Continued Radicals}
In this section we shall obtain an inequality similar to the one in Theorem \ref{ineq2} for the more general sequence $\{v_n\}$ defined in (\ref{genrad}). This will yield Theorem \ref{ineq2} as a special case. The first thing to observe is this: since we are considering general right radicals, we may well assume, with no loss of generality, that $a_n>0$ for all $n$. Let us give an example to illustrate this point. Suppose $a_1, a_3, a_4>0$ and $a_2=0$. Clearly, we may write
\[
\sqrt[r_1]{a_{1}+\sqrt[r_2]{a_2+\sqrt[r_3]{a_3+\sqrt[r_4]{a_4}+\cdots}}}=\sqrt[r_1']{a_1'+\sqrt[r_2']{a_2'+\sqrt[r_3']{a_3'+\cdots}}},
\]
where $r_1'=r_1$, $a_1'=a_1$, $r_2'=r_2r_3$, $a_2'=a_3$, and for $i \geqslant 3$, $r_i'=r_{i+1}'$, $a_i'=a_{i+1}$. It is evident that this procedure may be performed suitably to eliminate all zero inputs from consideration without changing the ``value" of the radical. So, henceforth we will assume that $a_n>0$ for all $n$.

 It is an interesting fact that continued right radicals have a triangular-array type structure when it comes to computation. Indeed one can construct a triangular array $\{t_{i,n}\}_{1\leqslant i\leqslant n \text{, }n\geqslant1}$ where
\begin{equation}
t_{i,n}=\sqrt[r_{n+1-i}]{a_{n+1-i}+\cdots+\sqrt[r_n]{a_n}},
\end{equation}
so that the diagonals $\{t_{n,n}\}$ form the sequence of the $n^{th}$ approximants. To compute $t_{n,n}$ one starts with $t_{1,n}=\sqrt[r_n]{a_n}$ and successively computes $t_{i,n}=\sqrt[r_{n+1-i}]{a_{n+1-i}+t_{i-1,n}}$ for $i=2,\ldots,n$. To relate $v_n=t_{n,n}$ to $t_{1,n}$, we need a device to invert this sequence of operations. This motivates us to introduce certain \emph{``denesting"} functions as follows. Let $f_0(y):=y$ and for $k \geqslant 1$ recursively define 
\begin{equation}
f_{k}(y):=f_{k-1}(y)^{r_k}-a_k.
\end{equation}
The denesting functions $f_j$ map the diagonal term $t_{n,n}$ to the preceding terms of the $n^{th}$ row of the array, i.e. one has $f_j(t_{n,n})=t_{n-j,n}$ for $j=1,\ldots, n-1$. With these functions in hand, we are able to state the following generalization of the inequality in Theorem \ref{ineq2}:  

\begin{theorem}\label{ineq3}
For $n\geqslant 1$,
\begin{equation}
v_{n+1}-v_n \leqslant \frac{\sqrt[r_{n+1}]{a_{n+1}}}{\prod\limits_{i=1}^{n}r_i (f_{i-1}(v_n))^{r_i-1}}.
\end{equation}
\end{theorem}
In order to prove Theorem \ref{ineq3} we need the following key identity. 
\begin{lemma}\label{keylem}
For $n\geqslant 1$,
\begin{equation}\label{id1}
v_{n+1}-v_n=\frac{\sqrt[r_{n+1}]{a_{n+1}}}{\prod\limits_{i=1}^{n}(\sum\limits_{j=0}^{r_i-1}f_{i-1}(v_{n+1})^j f_{i-1}(v_n)^{r_i-1-j})}.
\end{equation}
\end{lemma}
\begin{proof}
As already observed, we have
\begin{equation}\label{l1}
f_j(v_n)=t_{n-j,n}=\sqrt[r_{j+1}]{a_{j+1}+\cdots+\sqrt[r_n]{a_n}} \, \text{, for $j=1,\ldots, n-1$.}
\end{equation}
Similarly,
\[
 f_{n-1}(v_{n+1})=\sqrt[r_n]{a_n+\sqrt[r_{n+1}]{a_{n+1}}}.
\]
Therefore,
\[
f_{n-1}(v_{n+1})^{r_n}-f_{n-1}(v_n)^{r_n}=\sqrt[r_{n+1}]{a_{n+1}},
\]
i.e.,
\begin{equation}\label{id2}
(f_{n-1}(v_{n+1})-f_{n-1}(v_n))(\sum\limits_{j=0}^{r_n-1}f_{n-1}(v_{n+1})^j f_{n-1}(v_n)^{r_n-1-j})=\sqrt[r_{n+1}]{a_{n+1}}.
\end{equation}
Now using the identity 
\[
f_i(x)-f_i(y)=f_{i-1}(x)^{r_i}-f_{i-1}(y)^{r_i}=(f_{i-1}(x)-f_{i-1}(y))(\sum\limits_{j=0}^{r_i-1}f_{i-1}(x)^j f_{i-1}(y)^{r_i-1-j})
\]
repeatedly to the LHS of (\ref{id2}) we get
\[
 (f_0(v_{n+1})-f_0(v_n))\prod\limits_{i=1}^{n}(\sum\limits_{j=0}^{r_i-1}f_{i-1}(v_{n+1})^j f_{i-1}(v_n)^{r_i-1-j})=\sqrt[r_{n+1}]{a_{n+1}}.
\]
Noting that $f_0(v_{n+1})-f_0(v_n)=v_{n+1}-v_n$ one finally obtains the desired identity.
\end{proof}
\noindent
\begin{proof}[Proof of Theorem \ref{ineq3}]
Note that
\[
f_{k}'(y)=r_k f_{k-1}(y)^{r_k-1}f_{k-1}'(y).
\] 
Using this repeatedly along with the facts that $f_0(y)=y$ and $f_i$ is positive in $(v_i,\infty)$, we first note for each $k\geqslant 1$ that $f_k$ is strictly increasing in $(v_{k-1},0)$, where $v_0:=0$. Also because of positive input, $\{v_n\}$ is a strictly increasing sequence of positive reals. Therefore, $f_i(v_{n+1})\geqslant f_i(v_n)$ for each $i=0,1,\ldots,n-1$ (actually $f_i(v_{n+1})> f_i(v_n)$, but that does not matter since we will only be interested in the limiting behavior) and hence,
\[
 \sum\limits_{j=0}^{r_i-1}f_{i-1}(v_{n+1})^j f_{i-1}(v_n)^{r_i-1-j} \geqslant r_i f_{i-1}(v_n)^{r_i-1}.
\]
Combining this observation with identity (\ref{id1}) the proof is complete.
\end{proof}

Theorem \ref{ineq3} yields the following weaker inequality which can be considered as a generalization of Corollary \ref{ineq1}.
\begin{cor}\label{ineq5}
For $n\geqslant 1$,
\begin{equation}
v_{n+1}-v_n \leqslant \frac{\sqrt[r_{n+1}]{a_{n+1}}}{\prod\limits_{i=1}^{n}r_i a_i^{\frac{r_i-1}{r_i}}}.
\end{equation}
\end{cor}
\begin{proof}
The identities in (\ref{l1}) imply that for $i=0,1,\ldots,n-1$ one has $f_i(v_n) > \sqrt[r_{i+1}]{a_{i+1}}$. Using this in Theorem \ref{ineq3} completes the proof. 
\end{proof}
Specializing to the case $r_1=r_2=\cdots=2$ one recovers Theorem \ref{ineq2} and Corollary \ref{ineq1} immediately. It is a curious fact that for $r_i=1$, the radicals become infinite series, and the  inequality of Corollary \ref{ineq5} reduces to the (sharpest possible!) statement that $a_{n+1}\leqslant a_{n+1}$.

 Following Theorem \ref{ineq3}, a sufficient condition for the convergence of $\{v_n\}$ can be given along the lines of Proposition \ref{pol2}.
\begin{theorem}\label{suffcond1}
The sequence $\{v_n\}$ converges if the series
\[
S=\sum\limits_{n=1}^{\infty}\sqrt[r_{n+1}]{a_{n+1}}\left(\prod\limits_{i=1}^{n}r_i (f_{i-1}(v_n))^{r_i-1}\right)^{-1}
\]
converges.
\end{theorem}
\begin{proof}
We show that under the hypothesis $\{v_n\}$ is Cauchy. Let us denote the partial sums of the series $S$ by $S_n$ with $S_0:=0$. Then we have for $m>n\geqslant 1$,
\begin{align*}
0\leqslant v_m-v_n &=\sum\limits_{k=n}^{m-1}(v_{k+1}-v_{k})\\
           &\leqslant \sum\limits_{k=n}^{m-1}\sqrt[r_{k+1}]{a_{k+1}}\left(\prod\limits_{i=1}^{k}r_i (f_{i-1}(v_k))^{r_i-1}\right)^{-1}\\
           &\leqslant \sum\limits_{k=n}^{\infty}\sqrt[r_{k+1}]{a_{k+1}}\left(\prod\limits_{i=1}^{k}r_i (f_{i-1}(v_k))^{r_i-1}\right)^{-1}\\
           &=S-S_{n-1},
\end{align*}
which is the tail sum of the convergent series $S$. Hence, given any $\varepsilon >0$, we can choose $N\geqslant 1$ sufficiently large such that for all $n\geqslant N$ we have $S-S_{n-1}<\varepsilon$. Then for all $m>n\geqslant N$, we have
\[
0\leqslant v_m-v_n\leqslant S-S_{n-1}<\varepsilon.
\]
This completes the proof.
\end{proof}
 
\section{Further Generalizations}
\subsection{A straightforward generalization}
Theorem \ref{ineq3} (and Lemma \ref{keylem}, Corollary \ref{ineq5} etc.) can be generalized to include a more general class of continued right radicals:
\begin{equation}\label{phimoregen}
w_n=b_1\sqrt[r_1]{a_{1}+b_2\sqrt[r_2]{a_2+\cdots+b_n\sqrt[r_n]{a_n}}},
\end{equation}
where the $a_j$, $b_j$ are positive reals (taking each $b_j=1$ gives back the previous ones). Note that one can reduce $w_n$ to the form $v_n$ by bringing the $b_j$ ``inside", i.e. one can write
\begin{equation}
w_n=\sqrt[r_1]{c_1+\sqrt[r_2]{c_2+\cdots+\sqrt[r_n]c_n}},
\end{equation}
where $c_i:=b_1^{r_1\cdots r_i}b_2^{r_2\cdots r_i}\cdots b_i^{r_i} a_i$. Now Theorem \ref{ineq3} (Lemma \ref{keylem}, Corollary \ref{ineq5}) can be readily applied to $w_n$ to obtain an inequality. For computational purposes the following recurrence relation is useful:
\begin{equation}
c_i=a_i\left(\frac{c_{i-1}b_i}{a_{i-1}}\right)^{r_i} \text{, for $i\geqslant 1$, and we define $a_0=b_0=1$}.
\end{equation} 

Alternatively, one can directly follow the proof of Theorem \ref{ineq3} and define the denesting functions $f_j$ suitably. To elaborate, one needs to define $f_0(y)=y$ and for $k \geqslant 1$ 
\[
f_k(y):=\left(\frac{f_{k-1}(y)}{b_k}\right)^{r_k}-a_k,
\]
and proceed as in Lemma \ref{keylem}. In any case, one will be led to the following
\begin{theorem}\label{ineq6}
For $n\geqslant 1$ we have
\begin{equation}\label{gen1}
w_{n+1}-w_n\leqslant\frac{b_{n+1}\sqrt[r_{n+1}]{a_{n+1}}\prod\limits_{i=1}^{n}b_i^{r_i}}{\prod\limits_{i=1}^{n}r_i (f_{i-1}(w_n))^{r_i-1}}.
\end{equation}
\end{theorem}
The corresponding P\'{o}lya-Szeg\H{o} type version is
\begin{cor}\label{ineq7}
For $n\geqslant 1$,
\begin{equation}\label{gen2}
w_{n+1}-w_n\leqslant\frac{b_{n+1}\sqrt[r_{n+1}]{a_{n+1}}\prod\limits_{i=1}^{n}b_i}{\prod\limits_{i=1}^{n}r_i a_i^{\frac{r_i-1}{r_i}}}.
\end{equation}
\end{cor}
\subsection{Generalization to continued power forms } 
Continued radicals generalize to continued power forms
\begin{equation}\label{genpow}
(a_1+(a_2+(a_3+\cdots)^{p_3})^{p_2})^{p_1},
\end{equation}
where the $p_j$ are allowed to take \emph{any} non-zero real value (and the $a_j$ are positive). Herschfeld \cite{h'feld} stated a generalized version of his convergence criterion for such power forms when $p_j \in (0,1]$.
\begin{theorem}[Herschfeld, 1935]\label{h'feld2}
Let
\begin{equation}\label{phigen}
t_n=(a_1+(a_2+\cdots +a_n^{p_n})^{p_2})^{p_1},
\end{equation}
where $p_j \in (0,1]$ for each $j\geqslant 1$ and suppose that the series
\[
S=\sum\limits_{i=1}^{\infty}p_1\cdots p_i
\]
converges. Then the sequence $\{t_n\}$ converges if and only if
\[
\limsup_{n\rightarrow\infty}a_n^{p_1\cdots p_n}<\infty.
\]
\end{theorem}
Jones \cite{jones1} developed a ratio test type sufficient condition for the convergence of continued power forms with $p_1=p_2=\cdots=p>1$. One expects that approximation inequalities similar to the one in Theorem \ref{ineq3} can be found for these general power forms. However, note that the proof of Lemma \ref{keylem} works only for positive integral $r_j$. It is possible to give a proof of Theorem \ref{ineq3} which bypasses the identity of Lemma \ref{keylem} and generalizes to continued power forms. Consider the approximants $t_n$ of the continued power form (\ref{genpow}). Define the denesting functions $f_j$ recursively by setting $f_0(y):=y$ and for $k\geqslant 1$,
\[
f_k(y):=(f_{k-1}(y))^\frac{1}{p_k}  -a_k. 
\]
We first define a function $g_n(x)$ by substituting $(a_n+x)$ in place of $a_n$ in $t_n$, where $x\in (-a_n,\infty)$. Note then that $g_n(0)=t_n$ and $g_n(a_{n+1}^{p_{n+1}})=t_{n+1}$. Clearly, $g_n$ is a differentiable function of $x$ and
\begin{align}\label{expr_deriv}
g_n'(x) &= p_1 f_0(g_n(x))^{\frac{p_1-1}{p_1}} \times p_2 f_1(g_n(x))^{\frac{p_2-1}{p_2}}\times\ldots\times p_n f_{n-1}(g_n(x))^{\frac{p_n-1}{p_n}}\\ \nonumber
&= \prod_{i=1}^n p_i f_{i-1}(g_n(x))^{\frac{p_i-1}{p_i}}.
\end{align}
By the mean value theorem, for some $\xi \in (0, a_{n+1}^{p_{n+1}})$, we can write
\begin{align}\label{expr_diff}
t_{n+1}-t_n &= g(a_{n+1}^{p_{n+1}})-g_n(0)\\ \nonumber
			&= (a_{n+1}^{p_{n+1}}-0)g_n'(\xi) \\ \nonumber
			&= a_{n+1}^{p_{n+1}}\prod_{i=1}^n p_i f_{i-1}(g_n(\xi))^{\frac{p_i-1}{p_i}}.
\end{align}
So, now we need to bound $f_{i-1}(g_n(\xi))^{\frac{p_i-1}{p_i}}$ suitably. For simplicity of exposition we shall do here only the case where all $p_i>0$, although the other cases can be handled similarly. Then from the expression (\ref{expr_deriv}) we find that $g_n$ is strictly increasing in its domain. Therefore,
\begin{equation}\label{bounds_for_g_n}
t_n=g_n(0)<g_n(\xi)<g_n(a_{n+1}^{p_{n+1}})=t_{n+1}.
\end{equation}
Now,
\[
f_{k}'(y)=\frac{1}{p_k}f_{k-1}(y)^{\frac{1}{p_k}-1}f_{k-1}'(y).
\]
Using this repeatedly with the facts that $f_0(y)=y$, $f_i(y)$ is positive in $(t_i,\infty)$ and each $p_k>0$, we conclude that $f_k$, for each $k\geqslant 1$, is strictly increasing in $(t_{k-1},\infty)$, where $t_0:=0$. Therefore, since $t_k$ increases with $k$, we have from (\ref{bounds_for_g_n}) that for each $i=1,\ldots, n$,
\begin{equation}
f_{i-1}(t_n)<f_{i-1}(g_n(\xi))<f_{i-1}(t_{n+1}).
\end{equation}
Plugging these bounds into (\ref{expr_diff}) we obtain the following generalization of Theorem \ref{ineq3}:
\begin{theorem}\label{genpow_thm}
Consider the continued power form (\ref{genpow}) with all $p_i>0$. Then we have for each $n\geqslant 1$
\[
t_{n+1}-t_n \leqslant a_{n+1}^{p_{n+1}}\prod_{\substack{
1 \leqslant i \leqslant n\\
 p_i \leqslant 1 }} p_i f_{i-1}(t_n)^{\frac{p_i-1}{p_i}}\prod_{\substack{
1\leqslant i \leqslant n\\
\, p_i > 1}} p_i f_{i-1}(t_{n+1})^{\frac{p_i-1}{p_i}}.
\]
\end{theorem}
\begin{rem}
Indeed, on taking $p_i = \frac{1}{r_i}$, we recover Theorem \ref{ineq3} from Theorem \ref{genpow_thm}.
\end{rem}
Note that the bound in Theorem \ref{genpow_thm} involves $t_{n+1}$ if some $p_j>1$. As such it is not that useful. However, if all $p_j\in (0,1]$, then we can readily use it. In that case, note that, using the obvious lower bound $f_{i-1}(t_n)\geqslant a_i^{p_i}$, we may obtain a P\'{o}lya-Szeg\H{o} type version (generalization of Corollary \ref{ineq5}).
\begin{cor}
If $p_k\in(0,1]$ for each $k\geqslant 1$, then for each $n\geqslant 1$,
\[
t_{n+1}-t_{n}\leqslant a_{n+1}^{p_{n+1}}\prod_{i=1}^n p_i a_i^{p_i-1}.
\]
\end{cor}
\begin{rem}
The case where each $p_i<0$ also merits mention. For example, when each $p_i=-1$, we get a continued fraction. Arguments leading to Theorem \ref{genpow_thm} can be modified suitably to obtain approximation inequalities in this case. Owing to the $\prod p_j$ factor in (\ref{expr_diff}), there will be two types of inequalities in this case, one for $n$ even, the other for $n$ odd. We leave the details to the reader.
\end{rem}
\section{Some Examples}
In this section we provide some examples to illustrate the use of the inequalities derived in the earlier sections. Besides proving convergence of the radicals in these examples, we obtain the rates of convergence along the way.
\begin{exm}
Consider the case $r_1=r_2=\cdots=r\geqslant2$. Let $a_1=a_2=\cdots=a>0$, $b_1=b_2=\cdots=b>0$. Then we have the following infinite radical
\begin{equation}
b\sqrt[r]{a+b\sqrt[r]{a+\cdots}}.
\end{equation}
This type of radicals with $r=2$ have been considered in \cite{zimho}, where the authors derive many interesting properties of these radicals (e.g., representation of the rationals by such radicals). Using inequality (\ref{gen2}) we have
\[ 
0\leqslant w_{n+1}-w_n\leqslant cs^n \text{, with $c=ba^{\frac{1}{r}}$ and $s=\frac{b}{ra^{\frac{r-1}{r}}}$}.
\] 
Suppose $a,b,r$ are such that $s<1$. So, we have for $m > n\geqslant1$,
\begin{align*}
0\leqslant w_m-w_n &=\sum\limits_{k=n}^{m-1}(w_{k+1}-w_{k})\\
         &\leqslant c\sum\limits_{k=n}^{m-1} s^k\\
         &< cs^n\sum\limits_{k=0}^{\infty} s^k=\frac{cs^n}{1-s}.      
\end{align*}
So, $\{w_n\}$ is Cauchy and hence convergent. Denoting the limit by $w$ and letting $m\rightarrow \infty$ we obtain
\[
 0\leqslant w-w_n\leqslant\frac{cs^n}{1-s}.
\]
Thus the convergence rate is at least geometric. 
\end{exm}

\begin{exm}\label{exm:2}
Let us consider the radical
\begin{equation}
1+\sqrt{2+\sqrt[3]{3+\sqrt[4]{4+\cdots}}}.
\end{equation}
Here $r_n=a_n=n$. From Corollary \ref{ineq5} we have
\[
0\leqslant v_{n+1}-v_n\leqslant \frac{(n+1)^{\frac{1}{n+1}}}{n!\prod\limits_{i=1}^n i^{1-\frac{1}{i}}}=\frac{\prod\limits_{i=1}^{n+1}i^{\frac{1}{i}}}{(n!)^2}\leqslant\frac{(3^{\frac{1}{3}})^{n+1}}{(n!)^2}.
\]
For $n\geqslant 2$ one may use the weaker estimate
\[
v_{n+1}-v_n\leqslant\frac{3^{\frac{1}{3}}}{n!}.
\]
Using this and the well-known estimate $e-\sum\limits_{k=0}^n \frac{1}{k!}<\frac{1}{n (n!)}$ (see, for example, \cite {rudin}), it is an easy exercise to show that for $n\geqslant 2$,
\[
0\leqslant v-v_n\leqslant\frac{3^{\frac{1}{3}}}{(n-1)((n-1)!)}.
\]
A slightly modified version of the above radical is the interesting looking
\begin{equation}
1+2\sqrt{2+3\sqrt[3]{3+4\sqrt[4]{4+\cdots}}}.
\end{equation}
For this radical $a_n=b_n=r_n=n$ and inequality (\ref{gen2}) yields
\[
0\leqslant w_{n+1}-w_{n}\leqslant \frac{(n+1)!(n+1)^{\frac{1}{n+1}}}{n!\prod\limits_{i=1}^n i^{1-\frac{1}{i}}}=\frac{(n+1)\prod\limits_{i=1}^{n+1}i^{\frac{1}{i}}}{n!}\leqslant\frac{(n+1)(3^{\frac{1}{3}})^{n+1}}{n!}.
\]
For $n\geqslant 4$, $n! > 2^n$. So, for such $n$ we have
\[
w_{n+1}-w_{n} < 3^{\frac{1}{3}}(n+1)\left(\frac{3^{\frac{1}{3}}}{2}\right)^n.
\]
Writing $s=\frac{3^{\frac{1}{3}}}{2}<1$, we have for $m>n\geqslant 4$,
\[
0\leqslant w_m-w_n < 3^{\frac{1}{3}}\sum\limits_{i=n}^{m-1}(i+1)s^i < 3^{\frac{1}{3}}\sum\limits_{i=n}^{\infty}(i+1)s^i=\frac{2s^{n+1}}{(1-s)^2}(1+n(1-s)).
\]
Hence, for $n\geqslant 4$,
\[
0\leqslant w-w_n \leqslant \frac{2s^{n+1}}{(1-s)^2}(1+n(1-s)).
\]
\end{exm}
\begin{rem}
The estimates in Example \ref{exm:2} are somewhat crude and used for illustrative purposes only. For example, we have used the well known fact that the sequence $\{\sqrt[n]{n}\}$ achieves its maximum at $n=3$. Better estimates are available, e.g., for all $n\geqslant 1$ one has
\[
n^{\frac{1}{n}}<1+\frac{1}{\sqrt{n}}.
\]
\end{rem}
Our last example demonstrates that on many occasions the bounds provided by the generalized P\'{o}lya-Szeg\H{o} type inequality  (\ref{gen2}) are not quite sharp and one needs to invoke the stronger inequality of Theorem \ref{ineq7} in order to obtain a non-trivial bound in such cases.  

\begin{exm}[The Ramanujan radical]
  Consider the following radical \cite{jims}
\begin{equation}
\sqrt{1+2\sqrt{1+3\sqrt{1+\cdots}}}.
\end{equation}
In 1911, the great Indian mathematician Srinivasa Ramanujan posed the problem of finding its value in the Journal of the Indian Mathematical Society. When no answer arrived after six months, Ramanujan published the solution. With an ingenious manipulation he ``showed" that the value of the above radical is $3$. Ramanujan's argument was incomplete as he did not address convergence issues, but the value is correct. For a rigorous proof see \cite{borbarra, h'feld, sury} and also \cite{brendt} for further commentary on this radical. In our notation this radical has $a_n=1$, $b_n=n$ and $r_n=2$. So, inequality (\ref{gen2}) yields
\[
w_{n+1}-w_{n}\leqslant \frac{(n+1)!}{2^n},
\]
which is clearly trivial. This failure is inherent in inequality (\ref{gen2}) itself. Heuristically one expects that whenever the $b_n$ are ``large" compared to the $a_n$, the bound provided by inequality (\ref{gen2}) will be less precise.

To obtain a non-trivial inequality for the Ramanujan radical we shall employ the stronger inequality of Theorem \ref{ineq7}.

Note that 
\begin{align*}
f_{i-1}(w_n)& =i\sqrt{1+(i+1)\sqrt{1+\cdots+n\sqrt{1}}} \\
& \geqslant i\sqrt{i\sqrt{i\cdots i\sqrt{1}}} \\
& = i^{1+\frac{1}{2}+\frac{1}{2^2}+\cdots+ \frac{1}{2^{n-i}}}\\
&= i^{2-\frac{1}{2^{n-i}}}.
\end{align*}
Therefore, Theorem \ref{ineq7} yields
\begin{align*}
0\leqslant w_{n+1}-w_n & \leqslant \frac{(n+1)\prod_{i=1}^{n}i^2}{2^n \prod_{i=1}^{n}i^{2-\frac{1}{2^{n-i}}}}\\
& = \frac{n+1}{2^n}\prod_{i=1}^{n}i^{\frac{1}{2^{n-i}}}.
\end{align*}
Denote the right hand side by $h_n$. Taking logarithms we have
\begin{align*}
\log h_n - \log(n+1) + n \log 2 &= \frac{1}{2^n}\sum_{i=1}^{n}2^i \log i\\
 & \leqslant \frac{\log n}{2^n}\sum_{i=1}^{n}2^i \\
& \leqslant 2\log n.
\end{align*}
This implies that
\[
h_n \leqslant \frac{n^2(n+1)}{2^n}.
\]
Putting all these together we obtain
\[
0 \leqslant w_{n+1}-w_n  \leqslant \frac{n^2(n+1)}{2^n}.
\]
Therefore, we need to bound the tail sum  $\sum_{i=n}^{\infty}\frac{i^2(i+1)}{2^i}$. Write $s=1/2$ and note that
\begin{align*}
\sum_{i=n}^{\infty}\frac{i^2(i+1)}{2^i} & = \sum_{i=n}^{\infty}i^2(i+1)s^i \\
& \leqslant \sum_{i=n}^{\infty}(i+3)(i+2)(i+1)s^i.
\end{align*}
Now, for a real variable $s$, we have
\[
(i+3)(i+2)(i+1)s^i=\frac{d^3}{ds^3}s^{i+3},
\]
and when $|s|<1$, we can interchange differentiation and summation. This yields
\begin{align*}
\sum_{i=n}^{\infty}(i+3)(i+2)(i+1)s^i & = \sum_{i=n}^{\infty}\frac{d^3}{ds^3}s^{i+3}\\
&= \frac{d^3}{ds^3}\sum_{i=n}^{\infty}s^{i+3} \\
&= \frac{d^3}{ds^3}\frac{s^{n+3}}{1-s}.
\end{align*}
Using Liebnitz rule it is easy to show that
\[
\frac{d^3}{ds^3}\frac{s^{n+3}}{1-s}=\frac{n^3s^n}{1-s}(1+o(1)).
\]
Finally, we have
\[
0\leqslant w_m-w_n < \sum_{i=n}^{\infty}\frac{i^2(i+1)}{2^i}\leqslant \frac{n^3}{2^{n-1}}(1+o(1)),
\]
which proves the convergence of the Ramanujan radical and gives an estimate for the rate of convergence. We summarize this as
\[
0\leqslant 3-\sqrt{1+2\sqrt{1+\cdots+n\sqrt{1}}} \leqslant \frac{n^3}{2^{n-1}}(1+o(1)). 
\] 
\end{exm}
\section{Acknowledgements}
The author is thankful to Prof. Debapriya Sengupta for many useful discussions.

\end{document}